\documentclass[a4paper]{amsart}
\usepackage{amsmath,amsthm,amscd,amssymb}
\usepackage{color,graphicx}
\usepackage{atbegshi}
\usepackage[dvipdfm,pageanchor,pdfdisplaydoctitle=true,bookmarks=true,bookmarksnumbered=true,bookmarkstype=toc,colorlinks=true,linkcolor=red,citecolor=blue,urlcolor=blue,backref,bookmarks,breaklinks,linktocpage,
pdftitle={},%
pdfsubject={},%
pdfauthor={},%
pdfkeywords={}]{hyperref}
\hypersetup{pdfstartview={FitH -32768}}
\usepackage{enumerate}
\usepackage[varg]{txfonts}
\newcommand{\Z}{\mathbb{Z}}
\newcommand{\R}{\mathbb{R}}

\makeatletter\@ifundefined{coloneqq}{}\makeatother

\newtheorem{thm}{Theorem}[section]

\newtheorem{lem}[thm]{Lemma}
\newtheorem{prop}[thm]{Proposition}
\theoremstyle{definition}
\newtheorem{defn}[thm]{Definition}
\newtheorem{ex}[thm]{Example}
\newtheorem{rem}[thm]{Remark}
\newtheorem{notn}[thm]{Notation}
\newtheorem{algor}[thm]{Algorithm}

\date{}
\title{On a move reducing the genus of a knot diagram}
\author{Kenji Daikoku}
\author{Keiichi Sakai}
\address{Department of Mathematical Sciences, Faculty of Science, Shinshu University, Matsumoto, 390-8621 Japan.}
\email{10sa109j@shinshu-u.ac.jp}
\email{ksakai@math.shinshu-u.ac.jp}
\author{Masamichi Takase}
\address{Faculty of Science and Technology,
Seikei University, 3-3-1 Kichijoji-kitamachi, Musashino, Tokyo 180-8633, Japan.}
\email{takase@math.shinshu-u.ac.jp}

\begin{document}\sloppy
\begin{abstract}
For a knot diagram we introduce an operation 
which does not increase 
the genus of the diagram and does not change its 
representing knot type. 
We also describe a condition for this operation to 
certainly decrease the genus. 
The proof involves the study 
of a relation between 
the genus of a virtual knot diagram 
and the genus of a knotoid diagram, 
the former of which has been introduced by Stoimenow, Tchernov and Vdovina, and 
the latter by Turaev recently. 
Our operation has a simple interpretation in terms of Gauss codes and hence can easily be computer-implemented. 
\end{abstract}
\maketitle
\section{Introduction}
Seifert \cite{seifert} gave an algorithm to 
construct from a diagram $D$ of a knot $K$ 
an orientable surface 
bounded by the knot (see Definition~\ref{def:Seifert_algorithm}). 
We call the surface constructed by Seifert's algorithm 
\textit{the canonical Seifert surface for $D$}
and its genus \textit{the genus $g(D)$ of $D$}. 
The canonical genus $g_c(K)$ of a knot $K$ is defined to be the 
minimal genus of all possible diagrams. 
It is an important knot invariant and 
extensively studied (see for example \cite{stoimenow}). 
By definition the canonical genus of a knot $K$ gives 
an upper bound for the genus $g(K)$ of $K$, 
that is the minimum of genera of 
all possible Seifert surfaces for $K$. 

In this paper, 
we introduce an operation, called {\em the bridge-replacing move}, 
for a knot diagram  
which does not change its representing knot type and 
does not increase the genus of the diagram 
(see Definition~\ref{defn:daikoku_move} and Theorem~\ref{thm:main}). 
A necessary and sufficient condition for the operation 
to actually decrease the genus of the diagram 
is given in Proposition~\ref{prop:bypass}.

Our move is derived from Turaev's idea 
using the notion {\em knotoid} \cite{turaev}. 
{\em A knotoid diagram} $D^{\circ}$ is a diagram which differs 
from usual knot diagrams in that the underlying curve is an immersed 
interval rather than an immersed circle (see Definition~\ref{def:knotoid}). 
Turaev has constructed in \cite[\S2.5]{turaev} 
\textit{the canonical surface} 
for a given knotoid diagram by using an analogous procedure 
to the usual Seifert algorithm 
(see Definition~\ref{defn:surface_knotoid}).
{\em The genus of a knotoid diagram} is defined to be the genus 
of its canonical surface. 
In \cite[\S2.5]{turaev} it has been suggested that 
the study of knotoid diagrams can be used 
to obtain a good estimate for the genus of a knot. 
For, if we have a diagram $D$ of a non-alternating knot $K$ 
which contains a consecutive sequence $B$ of 
$k$ over- or under-crossing segments ($k\ge2$), 
then by removing $B$ from $D$ we can obtain the knotoid $D_B^\circ$
with genus $g(D_B^{\circ})\ge g(K)$.
In this context, we will prove in Proposition~\ref{prop:ineq}
that the inequality 
$g(D)\ge g(D^{\circ}_B)$ always holds,
which has not been explicitly proven in \cite[\S2.5]{turaev}. 
Moreover, the bridge-replacing move for $(D,B)$ 
(see Definition~\ref{defn:daikoku_move}) proves to 
produce the knot diagram $\tilde{D}_B$ with the property 
$g(\tilde{D}_B)=g(D^{\circ}_B)$.
This implies $g(D)\ge g(\tilde{D}_B)$, that is, 
the bridge-replacing move is indeed useful to obtain 
a better estimate for $g_c(K)$ 
than the one given by the initial diagram $D$.

For the proof, we will study in \S\ref{sect:main} a relation 
between the canonical surface associated to a knotoid diagram 
and the surface associated to a certain virtual knot diagram, 
which has been introduced by Stoimenow, Tchernov and Vdovina 
\cite{STV} through {\em Gauss diagrams}.

In \S\ref{s:ex}, we discuss an example which suggests that 
the bridge-replacing move is actually useful for the 
determination of the genus of a knot.

\section{Preliminaries}
\subsection{Knots}
Here we recall some well-known notions concerning 
classical knots in $\R^3$.

A {\em Seifert surface} for a knot $K$ is a connected, 
oriented surface $\Sigma$ embedded in $\R^3$ whose 
boundary $\partial\Sigma$ coincides with the knot $K$.
Recall {\em Seifert's algorithm} which produces a 
Seifert surface $\Sigma_D$ for a knot $K$ from a given 
diagram $D=D_K$ of $K$:

\begin{defn}[Seifert's algorithm \cite{seifert}]\label{def:Seifert_algorithm}
Given a diagram $D=D_K$ of a knot $K$, the 
{\em canonical Seifert surface} $\Sigma_D$ for $K$ is 
the surface obtained in the following way.
\begin{enumerate}[(i)]
\item Draw $D$ in $\R^2\times\{ 0\}$ and orient $D$ 
 in an arbitrary way.
\item Smooth all the crossings of $D$ as in 
 Figure~\ref{fig:smoothing} to obtain a disjoint union of 
 embedded circles in the plane (called {\em Seifert circles}).
\item Fill the Seifert circles with the disjoint disks 
 in $\R^3$.
\item Take the half-twisted band sums along the original 
 crossings (Figure~\ref{fig:band_sum}) of $D$.
\end{enumerate}
\end{defn}

\begin{figure}
\includegraphics{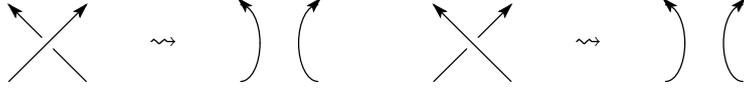}
\caption{Smoothing of the crossings}
\label{fig:smoothing}
\end{figure}

\begin{figure}
\includegraphics{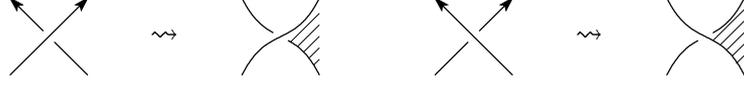}
\caption{Half-twisted band sum}
\label{fig:band_sum}
\end{figure}

The following is obvious.

\begin{lem}\label{lem:canonical_genus_knot_diagram}
Given a knot diagram $D=D_K$ with $n$ crossings, 
the genus of the canonical Seifert surface $\Sigma_D$ of 
$D$ is given by the formula
\begin{equation}\label{eq:canonical_genus_knot_diagram}
 g(\Sigma_D)=\frac{n-s_D+1}{2},
\end{equation}
where $s_D$ is the number of the Seifert circles obtained 
in the step (ii) in Definition~\ref{def:Seifert_algorithm}.
\end{lem}

According to Stoimenow, Tchernov and Vdovina \cite{STV}, 
for the {\em Gauss diagram} of a knot diagram $D$ we can 
construct a surface, which turns out to be homeomorphic to 
the canonical Seifert surface $\Sigma_D$, as explained below.

\begin{defn}\label{def:Gauss_diagram}
A {\em Gauss diagram} is an oriented circle equipped 
with some number $n$ of signed, oriented chords each of 
which connects distinct two points on the circle 
(all the $2n$ points are distinct with each other).
\end{defn}

To each knot diagram $D=D_K$ with $n$ crossings, we can 
assign the Gauss diagram $G_D$ with $n$ chords as follows:
\begin{enumerate}[(i)]
\item Connect the preimages of each crossing of $D$ by a chord.
\item Choose the orientation of each chord from the overpass 
 branch to the underpass one.
\item Give to each chord the sign $+$ or $-$ depending on 
 whether the corresponding crossing is positive (the left 
 crossing in Figure~\ref{fig:smoothing}) or negative, respectively.
\end{enumerate}

\begin{ex}\label{ex:trefoil_Gauss_diagram}
Figure~\ref{fig:trefoil} shows a knot diagram $D$ of the 
trefoil knot $3_1$ and its Gauss diagram.
The endpoint $\bar{i}$ of $G$ corresponds to the over-arc 
of the crossing $i$ of $D$.
\begin{figure}
\includegraphics{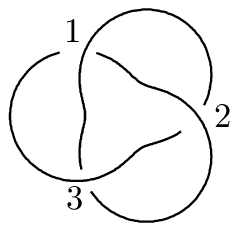}\qquad
\includegraphics{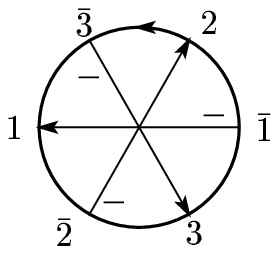}
\caption{A knot diagram $D$ of $3_1$ and its Gauss diagram $G$}
\label{fig:trefoil}
\end{figure}
\end{ex}

\begin{defn}\label{def:canonical_surface_virtual}
Let $G$ be a Gauss diagram.
The {\em canonical surface} $\Sigma_G$ of $G$ is the surface 
(with exactly one boundary component) obtained in the following way.
\begin{enumerate}[(i)]
\item Consider an annulus $A=S^1\times [0,1]$.
\item Take a band $B_c=[0,1]\times [0,1]$ for each chord $c$ of $G$.
\item For each chord $c$ and $\epsilon\in\{ 0,1\}$, glue 
 $\{\epsilon\}\times [0,1]\subset B_c$ to 
 $[x_{\epsilon}-\delta ,x_{\epsilon}+\delta ]\times\{0\}\subset \partial A$ 
 (where $x_0,x_1\in S^1=\R /\Z$ are the endpoints of $c$, 
 and $\delta >0$ is a sufficiently small number) so that 
 $\Sigma'_G:=A\cup \bigl(\bigcup_cB_c\bigr)$ is oriented.
\item Glue all the boundary components of $\Sigma'_G$ but 
 $S^1\times\{1\}\subset\partial A$ with disks to obtain an 
 oriented surface $\Sigma_G$ with one boundary component.
\end{enumerate}
\end{defn}

\begin{ex}
Figure~\ref{fig:trefoil_surface} shows $\Sigma'_G$ for the 
Gauss diagram $G$ of Figure~\ref{fig:trefoil}.
\begin{figure}
\includegraphics{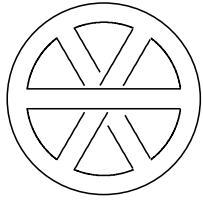}
\caption{$\Sigma'_G$ for the Gauss diagram $G$ of Figure~\ref{fig:trefoil}}
\label{fig:trefoil_surface}
\end{figure}
Attaching two disks along the boundary components of 
$\Sigma'_G$ other than $S^1\times\{ 1\}$ (the outermost circle), 
we obtain the canonical surface $\Sigma_G$ of $G$, which is 
homeomorphic to the canonical Seifert surface of $D$ shown in 
Figure~\ref{fig:trefoil} (as proven in Proposition~\ref{prop:surfaces_homeo}).
\end{ex}

\begin{rem}
The canonical surface for a Gauss diagram constructed by 
Stoimenow, Tchernov and Vdovina in \cite{STV} is a closed surface.
Removing a disk (attached along $S^1\times\{1\}$) from 
their surface, we obtain the surface defined in 
Definition~\ref{def:canonical_surface_virtual}.
\end{rem}

\begin{prop}[{\cite[Theorem~2.5]{STV}}]\label{prop:surfaces_homeo}
For a knot diagram $D$, the surfaces $\Sigma_D$ and 
$\Sigma_{G_D}$ are homeomorphic to each other.
\end{prop}

For the proof, we introduce the notion of a 
{\em cycle in a Gauss diagram} which corresponds to a 
Seifert circle in Seifert's algorithm.

\begin{defn}\label{def:cycle_in_Gauss_diagram}
Let $G$ be a Gauss diagram. 
A {\em cycle} in $G$ is the cycle obtained by repeating the 
following steps: 
starting from some endpoint $x$ of a chord $c$,
\begin{enumerate}[(i)]
\item go to the other endpoint of $c$ along $c$,
\item go to the next endpoint $y$ along the circle with respect to 
 its orientation,
\item from $y$, repeat (i) and (ii) above until coming back to $x$ 
 in the step (ii).
\end{enumerate}
We will denote a cycle as a cyclic sequence 
$\{ i_1,i_2,\dots ,i_{2k}\}$ of endpoints which appear 
in the above steps (i) and (ii).
\end{defn}

\begin{ex}\label{ex:cycle_Gauss_diagram}
The Gauss diagram $G$ of Figure~\ref{fig:trefoil} has two cycles
$\{\bar{1},1,\bar{2},2,\bar{3},3\}$ and
$\{1,\bar{1},2,\bar{2},3,\bar{3}\}$.
\end{ex}

\begin{lem}\label{lem:canonical_genus_Gauss_diagram}
Given a Gauss diagram $G$ with $n$ chords, the genus of the 
canonical surface $\Sigma_G$ of $G$ is given by the formula
\begin{equation}\label{eq:canonical_genus_Gauss_diagram}
 g(\Sigma_G)=\frac{n-s_G+1}{2},
\end{equation}
where $s_G$ is the number of the cycles in $G$.
\end{lem}

\begin{proof}
This formula is a consequence of the following facts:
a cycle corresponds to a boundary component of $\Sigma'_G$ 
to which we glue the boundary of a disk in the final step 
(iv) in Definition~\ref{def:canonical_surface_virtual}.
Thus the number of disks we glue in the step (iv) in 
Definition~\ref{def:canonical_surface_virtual} is equal 
to the number of cycles in $G$.
\end{proof}

\begin{proof}[Proof of Proposition~\ref{prop:surfaces_homeo}]
It is easy to observe that
\begin{itemize}
\item the number of the crossings of $D$ and that of the 
chords in $G_D$ are the same,
\item $s_D=s_{G_D}$ (under the notations in 
Lemmas~\ref{lem:canonical_genus_knot_diagram}, 
\ref{lem:canonical_genus_Gauss_diagram}).
\end{itemize}
These facts and formulas \eqref{eq:canonical_genus_knot_diagram}, 
\eqref{eq:canonical_genus_Gauss_diagram} imply that 
the genus of $\Sigma_D$ is equal to that of $\Sigma_{G_D}$.
This completes the proof, since both $\Sigma_D$ and 
$\Sigma_{G_D}$ have exactly one boundary component.
\end{proof}

\begin{rem}\label{rem:how_to_reduce_genus}
The formula \eqref{eq:canonical_genus_knot_diagram} 
suggests that, to construct a knot diagram of smaller genus 
from a given knot diagram, we need an operation which 
decreases the number of the crossings or increases the 
number of Seifert circles.
The formula \eqref{eq:canonical_genus_Gauss_diagram} gives 
a similar suggestion for Gauss diagrams.
\end{rem}

\subsection{Virtual knots and the Gauss diagrams}\label{sect:virtual}
\begin{defn}\label{def:virtual_knot_diagram}
A {\em virtual knot diagram} is a generic immersion 
$S^1\looparrowright\R^2$ with only transverse double points 
as its singularities, some of which are endowed with 
over- or under-crossing data but others are not 
(see Figure~\ref{fig:8_20_and_8_20oid}).
A crossing endowed with over- or under-crossing information 
is called a {\em real crossing}, and one without such 
information is called a {\em virtual crossing}.
\end{defn}

Seifert's algorithm cannot be applied to a virtual 
knot diagram as it is.
But the Gauss diagram $G_{D^{\bullet}}$ of a virtual 
knot diagram $D^{\bullet}$ can still be defined in the 
same way as explained after 
Definition~\ref{def:Gauss_diagram}, except that no 
chords are assigned to virtual crossings (see \cite[\S2.1]{STV}). 
Therefore in view of Proposition~\ref{prop:surfaces_homeo}, 
the following definition would be natural.

\begin{defn}[\cite{STV}]
Let $D^{\bullet}$ be a virtual knot diagram.
The {\em canonical surface} $\Sigma_{D^{\bullet}}$ 
of $D^{\bullet}$ is defined to be the canonical surface 
$\Sigma_{G_{D^{\bullet}}}$ of the Gauss diagram 
$G_{D^{\bullet}}$ of $D^{\bullet}$. 
The {\em genus} $g(D^{\bullet})$ of $D^{\bullet}$ is 
defined to be the genus of the canonical surface 
$\Sigma_{D^{\bullet}}$ of $D^{\bullet}$.
\end{defn}

\subsection{Turaev's Knotoids}
\begin{defn}[\cite{turaev}]\label{def:knotoid}
A {\em knotoid diagram} is a generic immersion 
$f:[0,1]\looparrowright\R^2$ which has only transverse 
double points (endowed with over- or under-crossing data) 
as its singularities.
The endpoints $f(0)$ and $f(1)$ are called respectively 
the {\em leg} and the {\em head} of the knotoid diagram.
\end{defn}

See the central diagram in Figure~\ref{fig:8_20_and_8_20oid} 
for an example of a knotoid diagram.
In the diagram, $P$ is the leg and $Q$ is the head.
In \cite{turaev} a surface is produced from a knotoid 
diagram $D^{\circ}$, in an analogous way to Seifert 
algorithm for knot diagrams.

\begin{defn}[\cite{turaev}]\label{defn:surface_knotoid}
The {\em canonical surface $\Sigma_{D^{\circ}}$ of a 
knotoid diagram $D^{\circ}$} is the surface obtained as follows.
First, draw $D^{\circ}$ in $\R^2\times\{0\}\subset\R^3$.
Then:
\begin{enumerate}[(i)]
\item Orient $D^{\circ}$ in an arbitrary way.
\item Smooth all the crossings of $D^{\circ}$ as in 
 Figure~\ref{fig:smoothing} to obtain a disjoint union 
 of embedded circles in the plane (also called 
 {\em Seifert circles}) and an embedded interval 
 (denoted by $J$).
 We call $J$ the {\em Seifert interval}.
 $J$ has the same endpoints $x,y$ as $D^{\circ}$.
\item Fill the Seifert circles with the disjoint disks 
 lying above $\R^2\times\{ 0\}$, and take a band 
 $J\times [0,1]$ lying below $\R^2\times\{0\}$ and 
 meeting $\R^2\times\{0\}$ along $J\times\{0\}\approx J$.
\item Take the half-twisted band sums along the original 
 crossings (Figure~\ref{fig:band_sum}) of $D^{\circ}$.
 The resulting surface $\Sigma_{D^{\circ}}$ is the 
 canonical surface of $D^{\circ}$.
\end{enumerate}
The boundary of $\Sigma_{D^{\circ}}$ is the union of 
$D^{\circ}$ with $J\times\{1\}\cup\{x,y\}\times [0,1]$.
\end{defn}

\begin{defn}\label{def:genus_knotoid}
The {\em genus} $g(D^{\circ})$ of a knotoid diagram 
$D^{\circ}$ is defined to be the genus of the canonical 
surface $\Sigma_{D^{\circ}}$ of $D^{\circ}$.
\end{defn}

\section{The bridge-replacing move}\label{sect:main}
The main operation of this paper, 
the {\em bridge-replacing move}, 
is introduced in this section.
It is clear by definition that this operation does not 
change the representing knot types.
The observation given in Remark~\ref{rem:how_to_reduce_genus} 
is a key in proving that this operation does not increase, 
and sometimes certainly decreases, the genus of a knot diagram.

\begin{defn}\label{defn:bridge}
An {\em over-bridge} (resp.\ {\em under-bridge}) 
{\em of length} $k$ {\em of a knot diagram} $D$ is a 
consecutive sequence $B$ of $k$ over-crossing (resp.\ 
under-crossing) segments of $D$ (see the leftmost diagram 
in Figure~\ref{fig:8_20_and_8_20oid}).
\end{defn}

\begin{notn}\label{notn}
Let $D$ be a diagram of a knot $K$ and
$B$ be an over- or under-bridge of $D$.
We denote by $D^{\circ}_B$ the knotoid diagram obtained
by removing (the interior of) $B$ from $D$.
We denote by $D^{\bullet}_B$ the virtual knot diagram obtained
from $D$ by turning each crossing along $B$ into a virtual crossing.
See Figure~\ref{fig:8_20_and_8_20oid}.
\end{notn}

\begin{prop}\label{prop:eq}
Under Notation~\ref{notn}, we have
\[
g(D^{\circ}_B)=g(D^\bullet_B).
\]
\end{prop}

\begin{proof}
If we denote by $\delta$ the number of disks 
that were glued in the final step (the step (iv) in 
Definition~\ref{def:canonical_surface_virtual}) of the 
construction of $\Sigma_{D^\bullet_B}$ and by $\gamma$ 
the number of chords in the Gauss diagram $G_{D^\bullet_B}$,
then the Euler characteristic $\chi(\Sigma_{D^\bullet_B})$ 
of $D^\bullet_B$ is equal to $\delta-\gamma$.
We easily see that
$\delta$ equals the number of Seifert circles
in Turaev's construction of the Seifert surface $\Sigma_{D^{\circ}_B}$
and that $\gamma$ equals the number of crossings of $D^{\circ}_B$.
Thus it is clear that $\Sigma_{D^{\circ}_B}$
has the same Euler characteristic as $\Sigma_{D^\bullet_B}$ and
hence is homeomorphic to $\Sigma_{D^\bullet_B}$.
\end{proof}

\begin{rem}\label{rem:eq}
Given a knotoid diagram $D^{\circ}$, we can construct
a virtual knot diagram by connecting the two endpoints
of $D^{\circ}$ with an arbitrary
path along which only virtual crossings occur;
then it is a natural idea to define
the Gauss diagram of $D^{\circ}$ to be the
Gauss diagram of such a virtual knot diagram.
According to Proposition~\ref{prop:eq}, it turns out that 
we can alternatively define the genus of a knotoid
diagram (Definition~\ref{def:genus_knotoid})
to be the genus of its Gauss diagram
(just as Stoimenow, Tchernov and Vdovina \cite{STV}
define the genus of a virtual knot diagram).
\end{rem}

\begin{prop}\label{prop:ineq}
Under Notation~\ref{notn}, we have
\[
g(D)\ge g(D^{\circ}_B).
\]
\end{prop}

\begin{proof}
By Proposition~\ref{prop:eq}, it suffices
to compare the genera of the Gauss diagrams
$G_D$ and of $G_{D^\bullet_B}$. Note that
$G_{D^\bullet_B}$ is obtained by removing from $G_D$
the chords corresponding to the crossings along $B$.
In Lemma~\ref{lem:one_by_one}, we will prove that the
genus of any Gauss diagram does not increase after a
removal of a chord.
This implies the result.
\end{proof}

\begin{lem}\label{lem:one_by_one}
Let $G$ be any Gauss diagram and $G'$ be the Gauss diagram
obtained by removing from $G$ a chord $c$.
Then we have $g(G)\ge g(G')$.

In more detail, if $c$ appears twice in a single cycle of 
$G$ (as Case (i) in Figure~\ref{fig:configuration}), then 
we have $g(G)=g(G')+1$ and hence $g(G)>g(G')$.
Otherwise we have $g(G)=g(G')$.
\end{lem}

\begin{proof}
In order to compare the genera of $G$ and of $G'$, we need only 
to know the change of the numbers of their chords and cycles 
in constructing $G'$ from $G$ by removing the chord $c$ (see 
the formula \eqref{eq:canonical_genus_Gauss_diagram} 
in Lemma~\ref{lem:canonical_genus_Gauss_diagram} 
and Remark~\ref{rem:how_to_reduce_genus}).

In the case where $c$ appears twice in a single cycle 
(say $\alpha$) of $G$, this $\alpha$ is of the form
\[
 \alpha =\{ w_1,a,b,w_2,b,a\}
\]
for $a,b\in\partial c$ and some words $w_1$, $w_2$ 
in the endpoints.
Let $f(w_i)$ (resp.\ $l(w_i)$) be the first (resp.\ last)
endpoint contained in $w_i$.
Then the points $a$, $b$, $f(w_i)$ and $l(w_i)$ ($i=1,2$) 
are located on the circle as Case (i) in 
Figure~\ref{fig:configuration}.
\begin{figure}
\includegraphics{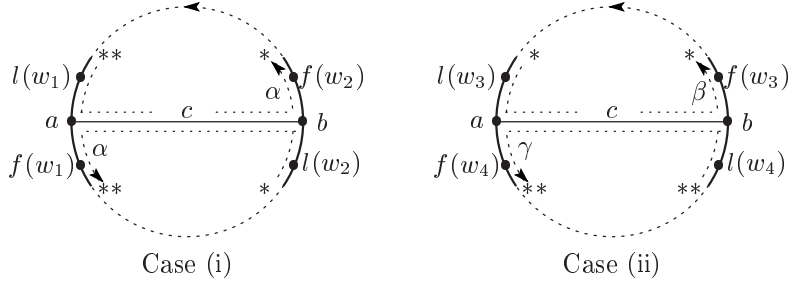}
\caption{The thickened curves are the parts of the circle of the Gauss diagram $G$.}
\label{fig:configuration}
\end{figure}
After the chord $c$ is removed (then the points 
$a$ and $b$ are also removed), the cycle $\alpha$ splits 
into two cycles $\{ w_1\}$ and $\{ w_2\}$, because $f(w_1)$ 
(resp.\ $f(w_2)$) appears just after $l(w_1)$ (resp.\ $l(w_2)$).
Thus the number of cycles increases by one. 
Since the number of chords decreases by one 
after $c$ is removed, 
we have $g(G)=g(G')+1$ by the formula 
\eqref{eq:canonical_genus_Gauss_diagram}.

Next consider the case where such a cycle as above does 
not exist. 
Then there are exactly two cycles 
(say $\beta =\{ w_3,a,b\}$ and $\gamma =\{ w_4,b,a\}$) 
in each of which $c$ appears exactly once.
Then $\beta$ and $\gamma$ are 
located as Case (ii) in Figure~\ref{fig:configuration}.
Thus after the removal of $c$, the cycle $\beta$ is 
unified with $\gamma$ by a band-sum along $c$, and the 
result is a single cycle $\{ w_3,w_4\}$.
Thus the number of cycles decreases by one.
Thus the formula \eqref{eq:canonical_genus_Gauss_diagram} 
implies that $g(G)=g(G')$.
\end{proof}

\begin{rem}
In \cite[\S1]{turaev}, it reads that 
``\textit{The study of knotoid diagrams leads to an elementary 
but possibly useful improvement of the standard Seifert estimate 
from above for the genus of knot.}''
However, it has not been rigorously proven in \cite{turaev}
that a knotoid diagram always gives
an estimate not worse
than ``the usual Seifert estimate'' with respect to the genus.
Propositions~\ref{prop:ineq} and \ref{prop:eq},
which deduce the inequality $g(D)\ge g(D^\circ_B)$, implement it. 
Notice that introducing the notion of the Gauss diagram
of a knotoid has been the main key
in the proof (see Remark~\ref{rem:eq}).
\end{rem}

\begin{rem}
We do not need to consider the removal of any 
``sub-bridge'' of $B$ since, by Lemma~\ref{lem:one_by_one}, 
the removal of the whole $B$ always gives a better estimate 
of the genus than any removal of a sub-bridge.
\end{rem}

Now we introduce the bridge-replacing move for a knot diagram
which does not increase (Theorem~\ref{thm:main}), and
in some cases certainly decrease (Proposition~\ref{prop:less}),
the genus of the diagram (cf.\ {\cite[\S2.5]{turaev}}).

\begin{defn}[the bridge-replacing move]\label{defn:daikoku_move}
Let $D$ be a diagram of a knot $K$ with
an over-bridge (resp.\ under-bridge) $B$.
Then we define an operation,
called \textit{the bridge-replacing move for $(D,B)$},
which replaces the over-bridge (resp.\ under-bridge) $B$
into another over-bridge (resp.\ under-bridge) $\tilde{B}$, as follows.

Take an orientation on the diagram $D$
and consider the (oriented) knotoid diagram $D^{\circ}_B$ by
removing $B$ from $D$ (see Figure~\ref{fig:8_20_and_8_20oid}).
Let $J$ be the oriented Seifert interval of $D^{\circ}_B$,
obtained by the smoothing process (the step (ii) in
Definition~\ref{defn:surface_knotoid}),
from the leg $P$ to
the head $Q$ (see Figure~\ref{fig:8_20oid_smoothed}) in $D^{\circ}_B$.

Now the new bridge $\tilde{B}$ from $Q$ to $P$ is
constructed just along $J$ so that it goes on the right (resp.\ left)
side of $J$ if at all possible,
and is allowed to overpass (resp.\ underpass)
only the arcs of which $J$ consists.
Thus the union $D^{\circ}_B\amalg\tilde{B}$
provides a new knot diagram, which we denote by $\tilde{D}_B$ (see Figure~\ref{fig:8_20_DM}).
Clearly $\tilde{D}_B$ represents the same knot type $K$ as $D$ does.
\end{defn}

\begin{rem}
In Definition~\ref{defn:daikoku_move}, in fact, 
it is not so important which side of $J$ the new bridge goes on 
(although it is essential that the new bridge does not 
cause crossings outside $J$). 
We have adopted the convention here so that we can treat 
over-bridges and under-bridges symmetrically. 
This systematic treatment will be convenient 
in Algorithm~\ref{alg}.
\end{rem}

\begin{figure}
\includegraphics{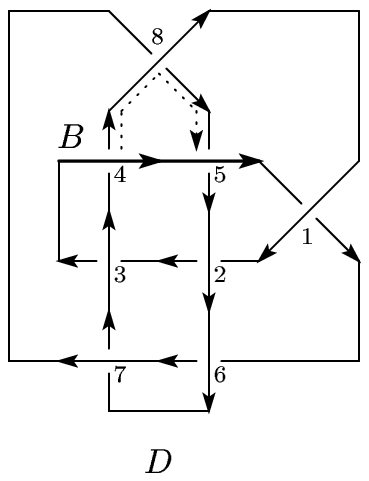}\quad
\includegraphics{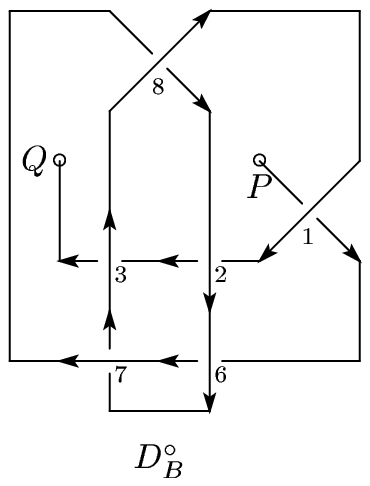}\quad
\includegraphics{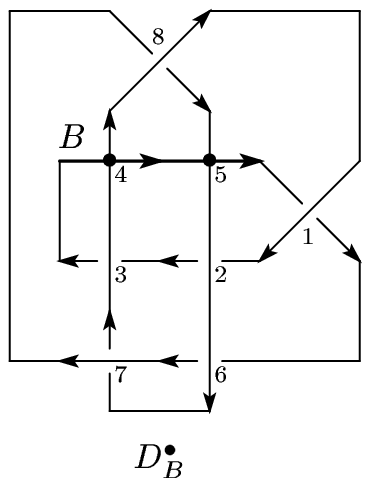}
\caption{A diagram $D$ of $8_{20}$, a knotoid diagram 
 $D^{\circ}_B$ obtained by removing an over-bridge 
 $B=(45)$ from $D$, and a diagram $D^{\bullet}_B$ of the 
 virtual knot obtained by turning the crossings $4$, $5$ 
 along $B$ into virtual crossings.
 The dotted arrow in $D$ is a {\em bypass} for $B$ 
 (see Proposition~\ref{prop:bypass}).}
\label{fig:8_20_and_8_20oid}
\end{figure}

\begin{figure}
\includegraphics{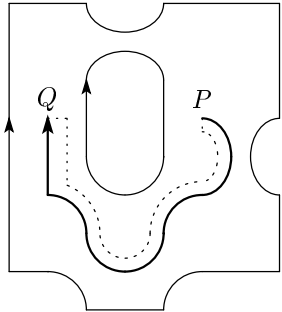}
\caption{The diagram obtained by smoothing all the 
 crossings of $D^{\circ}_B$.
 The thickened curve is the Seifert interval $J$.
 The dotted curve will be a guide for the new bridge 
 $\tilde{B}$ in the new diagram $\tilde{D}$.}
\label{fig:8_20oid_smoothed}
\end{figure}

\begin{figure}
\includegraphics{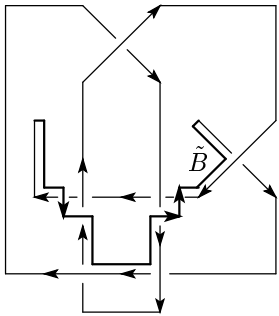}
\caption{The diagram $\tilde{D}$ of $8_{20}$ obtained by 
 replacing $B$ with $\tilde{B}$}
\label{fig:8_20_DM}
\end{figure}

Regarding the above bridge-replacing move,
our main claim is the following.

\begin{thm}\label{thm:main}
Let $D$ be a knot diagram with an over- or under-bridge $B$ and
$\tilde{D}_B$ be the knot diagram obtained by the bridge-replacing
move for $(D,B)$. Then we have
\[
g(D)\ge g(\tilde{D}_B).
\]
\end{thm}

To prove Theorem~\ref{thm:main},
we just need to combine Lemma~\ref{lem:main} below with
Propositions~\ref{prop:eq} and \ref{prop:ineq}.

\begin{lem}\label{lem:main}
Let $D$ be a knot diagram with an over- or under-bridge $B$.
Let $\tilde{D}_B$ be the knot diagram obtained by the bridge-replacing
move for $(D,B)$ and $D^{\circ}_B$ be the knotoid diagram
obtained by removing from $D$ the interior of $B$. Then we have
\[
g(\tilde{D}_B)=g(D^{\circ}_B).
\]
\end{lem}

\begin{proof}
We use the notations in Notation~\ref{notn} and
Definition~\ref{defn:daikoku_move}.

Since the new bridge $\tilde{B}$ passes only
the arcs which configures the Seifert interval $J$,
the Seifert algorithm for the new knot diagram
$\tilde{D}_B=D^{\circ}_B\amalg\tilde{B}$ provides the same result
that the Seifert algorithm for $D^{\circ}_B$ does,
except in the part with which $J$ is involved.

Assume that the new bridge $\tilde{B}$ has length $k$,
that is, $\tilde{B}$ passes $k$ times
the arcs which configures $J$.
Then, in applying the Seifert algorithm for the knot diagram 
$\tilde{D}_B$, considering the orientations of $J$ and 
of $\tilde{B}$, we have $(k+1)$ new Seifert circles 
(instead of $J$) in the part with which $J$ is involved.
This implies that
$g(\tilde{D}_B)$ is equal to $g(D^{\circ}_B)$.
\end{proof}

Next we study the condition for the bridge-replacing move 
to certainly decrease the genus.

An {\em over-bridge} (resp.\ {\em under-bridge}) of a 
Gauss diagram $G$ of length $k$ is a consecutive sequence of 
endpoints $v_0,\dots ,v_{k+1}$ on the oriented circle of $G$ 
such that $v_1,\dots ,v_k$ are the initial (resp.\ terminal) 
points of the chords of $G$.

\begin{prop}\label{prop:less}
Let $G$ be a Gauss diagram with an over- or under-bridge
$v_0,\dots ,v_{k+1}$, and let $G'$ be the Gauss diagram 
obtained by removing all the chords $c$ with $v_i\in \partial c$ 
for some $1\le i\le k$.
Then $g(G)>g(G')$ if and only if there is at least one cycle
$\alpha$ of $G$ of the form
\[
 \alpha =\{\dots v_iv_{i+1}\dots v_jv_{j+1}\dots\}
\]
for some $0\le i\ne j\le k$.
\end{prop}

\begin{proof}
This is a consequence of Lemma~\ref{lem:one_by_one}.
If such a cycle $\alpha$ as above exists, then Case (i) 
in Figure~\ref{fig:configuration} happens for 
at least one chord $c$ having its endpoint between 
$v_iv_{i+1}$ and $v_jv_{j+1}$, and hence the genus 
certainly decreases. 
Conversely, suppose any two of the portions $v_iv_{i+1}$
($0\le i\le k$) do not belong to the same cycle in $G$.
Then any chord $c$ with $v_i\in \partial c$ is placed 
as Case (ii) in Figure~\ref{fig:configuration}. 
Thus any removal of them does not decrease the genus.
\end{proof}

In the case of knot diagrams, the condition in 
Proposition~\ref{prop:less} can be stated as follows.

\begin{prop}\label{prop:bypass}
Let $D$ be a diagram of a knot $K$ with an over- or under-bridge $B$.
Then the bridge-replacing move for $B$ certainly decreases the genus 
if and only if there exists a collection of oriented arcs in $D$ 
which constitute, after the smoothing process 
(step (ii) in Seifert's algorithm, Definition~\ref{def:Seifert_algorithm}), 
a (well-oriented) path connecting two crossings included in $B$.
We call such a collection of arcs a {\em bypass}.
\end{prop}

For example in the leftmost diagram of 
Figure~\ref{fig:8_20_and_8_20oid}, 
the existence of the bypass (dotted) for 
the over-bridge $B$ (thickened) ensures that 
the bridge-replacing move for $(D,B)$ decrease the 
genus of the diagram.
In fact, the diagram shown in Figure~\ref{fig:8_20_DM} 
obtained from that in Figure~\ref{fig:8_20_and_8_20oid} 
by the bridge-replacing move for $B$ detects the genus $2$ of $8_{20}$.

\section{An example}\label{s:ex}
\begin{figure}
\includegraphics[scale=0.22]{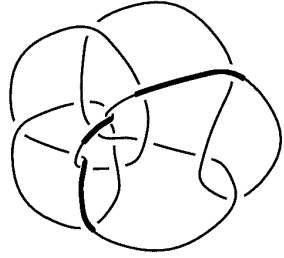}
\caption{The knot $16_{686716}$.
 This diagram has Dowker-Thistlethwaite code (DT code) 
 4 10 -26 -22 -18 2 20 -26 -32 -28 14 30 -6 -12 -8 24}
\label{fig:16-5}
\end{figure}

The knot $16_{686716}$, which is the $306917^{\rm th}$
non-alternating $16$ crossing knot in the Hoste-Thistlethwaite table \cite{h-t}, 
is referred to in Stoimenow's recent paper \cite[\S10.3]{stoimenow}
as a knot whose (canonical) 
genus has not been determined (to be whether $2$ or $3$) yet.
Here we demonstrate an approach to this problem using bridge-replacing moves.
Indeed its genus turns out to be $2$ (see Remark~\ref{rem:g=2}).

\begin{figure}
\includegraphics[scale=0.22]{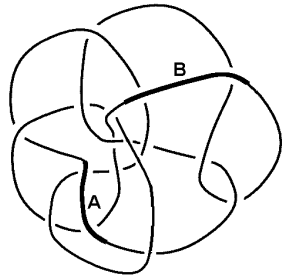}
\caption{A diagram with genus $5$}
\label{fig:17-5}
\end{figure}

In the knot diagram of $16_{686716}$ shown in Figure~\ref{fig:16-5}, 
that is drawn by 
Knotscape \cite{h-t} and has genus $5$,
we can find three over-bridges (thickened) each of which
has a favorable bypass (see Proposition~\ref{prop:bypass}).
In fact we can check that
the bridge-replacing move for one of the three bridges
produces a new diagram with genus $4$.
In this case, however, such a move for one bridge destroys
the bypasses of the other two bridges and apparently
we cannot perform further bridge-replacing moves.
To avoid this we precook the diagram into
the diagram shown in Figure~\ref{fig:17-5}.
Then we can perform for the diagram in Figure~\ref{fig:17-5}
bridge-replacing moves twice successively,
first with respect to the bridge $A$ and second with respect to $B$,
so that we obtain a diagram with genus $3$, but with $20$ crossings, of
the knot $16_{686716}$
(the number of the crossings of this diagram can be pared down to $18$ by the second Reidemeister move).

\begin{figure}
\includegraphics[scale=0.3]{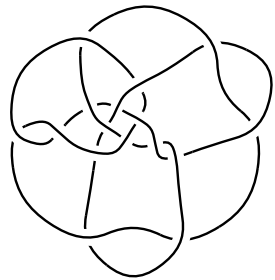}
\caption{A minimal crossing diagram with genus $3$ of $16_{686716}$.
Its DT code is -12 26 22 -14 28 -2 -20 30 -24 8 -32 -16 4 10 18 -6}
\label{fig:16-3}
\end{figure}

Now we can interpret the bridge-replacing move
in terms of Gauss codes, 
of which we give a detailed account in Appendix~\ref{appendix}.
By using the interpretation,
we develop a small Python program
which, for an inputted Gauss code,
performs all possible bridge-replacing moves
(and reduces excessive crossings in pairs
by the second Reidemeister moves).
With the aid of the program (and with some heuristic procedures),
we have found the diagram shown in Figure~\ref{fig:16-3} of the knot
$16_{686716}$,
that has genus $3$
and minimal (that is, $16$) crossings.

\begin{rem}\label{rem:g=2}
Inspired by Figures~\ref{fig:17-5} and \ref{fig:16-3}, Mikami Hirasawa found
Seifert surfaces of genus $2$ for the knot
$16_{686716}$ (by probing for ``compression disks'').
Figure~\ref{fig:minimal} depicts one of such (possibly non-canonical) Seifert
surfaces with genus $2$.
By an estimate via polynomial invariants, these surfaces turn out to attain the genus $2$ of the knot\footnote{This should be deduced also from a calculation of the knot Floer homology.}.
\begin{figure}
\includegraphics[scale=0.25]{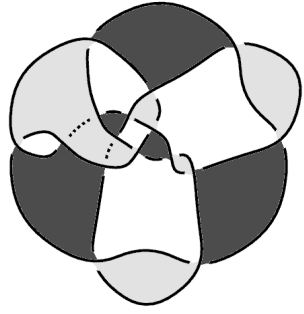}
\caption{A Seifert surface for $16_{686716}$ with genus $2$}
\label{fig:minimal}
\end{figure}
We do not know, however, if the canonical genus of
the knot can go down to $2$.
Namely, we do not know whether any of these 
Seifert surfaces for $16_{686716}$ with genus $2$ can be realized 
as a canonical one.
\end{rem}

\appendix
\section{The bridge-replacing move in terms of Gauss codes}\label{appendix}
We give an interpretation of 
the bridge-replacing move, introduced in \S\ref{sect:main},
in terms of {\em Gauss codes}.
We will follow the convention of \cite{kauffman} for Gauss codes.

\begin{defn}\label{def:Gauss_code}
A {\em unit} is a word of length three of the form 
``$\alpha\beta\gamma$,'' where
\begin{itemize}
\item ``$\alpha$'' is either the characters {\tt O} or {\tt U} 
 (which indicate ``over'' or ``under'' information),
\item ``$\beta$'' is a natural number (the label of a crossing),
\item ``$\gamma$'' is the sign $+$ or $-$ (the sign of a crossing).
\end{itemize}
A {\em Gauss code of length $n$} is a cyclic sequence of 
$2n$ units $X_1\dots X_{2n}$ such that:
\begin{enumerate}[(i)]
\item It includes $n$ distinct natural numbers
 and each number appears exactly twice.
\item If the number {\tt k} (which appears twice) appears 
 in some unit including the character {\tt U} (resp.\ {\tt O}), 
 then {\tt k} must appear in the other unit including {\tt O} 
 (resp.\ {\tt U}).
\item For each number {\tt k}, the two units including {\tt k} 
 have the same sign.
\end{enumerate}
\end{defn}

To an oriented knot diagram $D=D_K$ with $n$ crossings, 
 we can associate the Gauss code $X_D=X_1\dots X_{2n}$ 
 of length $n$ as follows:
label all the crossings as $1,2,\dots ,n$, and go along $D$ 
 according to the orientation (starting from an arbitrary point $b$).
We will meet each crossing exactly twice before we come back to $b$.
The $i$-th unit $X_i$ consists of
\begin{itemize}
\item the character {\tt O} or {\tt U} depending on whether 
 we overpass or underpass the $i$-th crossing respectively,
\item the natural number {\tt k} taken from the label of the 
 $i$-th crossing, and
\item the sign $+$ or $-$ depending on whether the $i$-th 
 crossing is positive or negative respectively.
\end{itemize}

The following definition is motivated by the notion of cycles 
 in Gauss diagrams (Definition~\ref{def:cycle_in_Gauss_diagram}).

\begin{defn}\label{def:cycle_in_code}
Let $C=X_1\dots X_{2n}$ be a Gauss code.
A {\em cycle} in $C$ is a cyclic sequence of units in $C$ 
obtained in the following steps: starting from some unit $X_i$,
\begin{enumerate}[(i)]
\item
 record the unit $X_j$ that contains the same natural number 
 as $X_i$,
\item
 record the next unit $X_{j+1}$ to $X_j$,
\item
 from $X_j$, repeat (i) and (ii) above until $X_i$ appears 
 again in the step (ii).
\end{enumerate}
We will denote a cycle as a cyclic sequence 
$X_{i_1}X_{i_2}\dots X_{i_{2k}}$ of units which appear in the 
above steps (i) and (ii).
\end{defn}

\begin{ex}
The Gauss code of the knot diagram of $3_1$ shown in 
Figure~\ref{fig:trefoil} is
\begin{center}
 $C=$ {\tt O1-U2-O3-U1-O2-U3-}.
\end{center}
This code has two cycles
{\tt O1-U1-O2-U2-O3-U3-} and
{\tt U1-O1-U2-O2-U3-O3-}
(compare them with cycles of $G$ given in 
Example~\ref{ex:cycle_Gauss_diagram}).
\end{ex}

\begin{rem}
A Gauss code for a virtual knot diagram can also be defined 
in the same way as explained after Definition~\ref{def:Gauss_code}, 
except that no units are assigned to virtual crossings.
\end{rem}

\begin{defn}\label{def:overbridge_GaussCode}
An {\em over-bridge} (resp.\ {\em under-bridge}) $B$ in a 
Gauss code $C=X_1\dots X_{2n}$ is a subsequence
$B=X_iX_{i+1}\dots X_j$ with all $X_p$ ($i\le p\le j$) including 
the character {\tt O} (resp.\ {\tt U}).
\end{defn}

Now we describe the bridge-replacing move, only for an over-bridge.
For an under-bridge, we need only to
interchange the letters \texttt{O} and \texttt{U} 
appearing in Algorithm~\ref{alg}.

\begin{algor}[the bridge-replacing move in terms of Gauss codes]\label{alg}
Given a Gauss code $C$ and an over-bridge $B$ in $C$,
we obtain the Gauss code $\tilde{C}$ by the following algorithm.
Let $n$ be the maximal natural number used in the code $C$.

\begin{enumerate}[(i)]
\item
Starting from the unit just before $B$,
search $C$ (cyclically) leftward and
find the first unit $X$ which does not contain
any number appearing in $B$.
\item
Remove all the unit which contain the numbers appearing in $B$,
so that we obtain the new code $C'$.
Note that $C'$ contains the unit $X$.
\item
Let $c$ be the cycle of $C'$
starting from the unit $X$.
\item
Starting from $X$, search $c$ (cyclically) leftward and
find all the sequences as \texttt{Om+Um+} or as \texttt{Um-Om-}
in $c$, where \texttt{m} is an arbitrary natural number;
record these natural numbers $a_1,a_2,\ldots,a_k$ in order (leftward from $X$).
\item
For $C'$,
insert the sequence
\texttt{O(n+1)-O(n+2)+O(n+3)-O(n+4)+$\mathtt{\cdots}$O(n+2k)+}
just after the unit $X$ in $C'$.
Furthermore,
\begin{itemize}
\item
if \texttt{O$\mathtt{a_j}$+U$\mathtt{a_j}$+} appears in $c$:\\
insert \texttt{U(n+2j-1)-} just after \texttt{U$\mathtt{a_j}$+} and
insert \texttt{U(n+2j)+} just before \texttt{O$\mathtt{a_j}$+}.
\item
if \texttt{U$\mathtt{a_j}$-O$\mathtt{a_j}$-} appears in $c$:\\
insert \texttt{U(n+2j-1)-} just after \texttt{O$\mathtt{a_j}$-} and
insert \texttt{U(n+2j)+} just before \texttt{U$\mathtt{a_j}$-}.
\end{itemize}
Then we denote the resultant code by $\tilde{C}$.
\end{enumerate}
\end{algor}

We can easily check that 
if $D$ is a knot diagram and its Gauss code is $C$, 
then $\tilde{C}$ obtained by Algorithm~\ref{alg} 
is nothing but the Gauss code of the diagram obtained by 
performing the bridge-replacing move for $(D,B)$, 
abusing the notation $B$ also for the bridge in the Gauss diagram 
corresponding to $B\subset C$.

\begin{ex}
Let $D$ be the knot diagram of $8_{20}$
as shown in Figure~\ref{fig:8_20_and_8_20oid}.
Its Gauss code $C=C_D$ is
\begin{center}
 {\tt O1+U2-U3+O4+O5-U1+U6-O7-U8-U5-O2-O6-U7-O3+U4+O8-}.
\end{center}
Consider an over-bridge $B=$ {\tt O4+O5-}.
\begin{enumerate}[(i)]
\item
 We find the unit $X=$ {\tt U3+} just before $B$, which does not
 contain neither {\tt 4} nor {\tt 5}.
\item
 Removing all the units containing {\tt 4} or {\tt 5},
 we obtain a code
	\begin{center}
	$C'=$ {\tt O1+U2-U3+U1+U6-O7-U8-O2-O6-U7-O3+O8-}
	\end{center}
 which indeed contains $X=$ {\tt U3+}.
 Here $C'$ corresponds to the knotoid diagram
 $D^{\circ}$ in Figure~\ref{fig:8_20_and_8_20oid}.
 The crossing labeled {\tt 3} is the first one
 which we meet after $Q$ along the knotoid diagram $D^{\circ}$.
\item
 The cycle $c$ of $C'$ starting from $X=$ {\tt U3+} is
	\begin{center}
	$c=$ {\tt U3+U1+O1+U2-O2-O6-U6-O7-O3+}.
	\end{center}
 This corresponds to the guide of the new bridge
 $\tilde{B}$ (Figure~\ref{fig:8_20oid_smoothed}).
\item
 In the cycle $c$, we find sequences
	\begin{center}
	{\tt U2-O2-, O3+U3+}
	\end{center}
 (notice that the codes are regarded as cyclic sequences).
 We put $\mathtt{a_1=3}$, $\mathtt{a_2=2}$
 (leftward order from $X$).
\item
 Since $n=8$ and $k=2$, we insert the sequence
	\begin{center}
	{\tt O9-O10+O11-O12+}
	\end{center}
 just after $X=$ {\tt U3+}.
 Moreover
	\begin{itemize}
	\item
	 since {\tt O$\mathtt{a_1}$+U$\mathtt{a_1}$+}
	 ($\mathtt{a_1=3}$) appears in $c$,
	 insert {\tt U9-} just after {\tt U3+} and
	 insert {\tt U10+} just before {\tt O3+},
	\end{itemize}
	\begin{itemize}
	\item
	 since {\tt U$\mathtt{a_2}$-O$\mathtt{a_2}$-}
	 ($\mathtt{a_2=2}$) appears in $c$,
	 insert {\tt U11-} just after {\tt O2-} and
	 insert {\tt U12+} just before {\tt U2-}.
	\end{itemize}
 These units correspond to the new crossings
 which are produced after the bridge $B$ is replaced
 by the new bridge $\tilde{B}$.
\end{enumerate}
In this way we obtain the new code
\begin{center}
 {\tt O1+U12+U2-U3+U9-O9-O10+O11-O12+U1+U6-O7-U8-O2-U11-O6-U7-U10+O3+O8-}
\end{center}
which represents the diagram (Figure~\ref{fig:8_20_DM}) 
obtained by performing the bridge-replacing move to $D$.
\end{ex}

\section*{Acknowledgments}
The authors would like to thank Kiyoshi Sasaki (Horrifunny Inc.)\ for his support
in coding the Python program used in \S\ref{s:ex} and for suggestions in writing Appendix~\ref{appendix}.
They are also grateful to Mikami Hirasawa for useful discussions.
KS is partially supported by the Grant-in-Aid for Research Activity Start-up, JSPS, Japan. 
MT is partially supported by the Grant-in-Aid for Scientific Research (C), JSPS, Japan.


\begin{thebibliography}{1}
\bibitem{h-t}
J.~Host, M.~Thistlethwaite, 
\emph{Knotscape}, 
an interactive program for the study of knots, 
available at \url{http://www.math.utk.edu/~morwen}

\bibitem{kauffman}
L.~H.~Kauffman, 
\emph{Virtual knot theory}, 
European J.\ Combin.\ \textbf{20} (1999), 
no.~7, 
663--690


\bibitem{seifert}
H.~Seifert, 
\emph{\"{U}ber das Geschlecht von Knoten}, 
Math.\ Ann.\ \textbf{110} (1935), 571--592 (in German)

\bibitem{stoimenow}
A.~Stoimenow, 
\emph{Knots of (canonical) genus two}, 
Fund.\ Math.\ \textbf{200} (2008), 1--67

\bibitem{STV}
A.~Stoimenow, V.~Tchernov, A.~Vdovina, 
\emph{The canonical genus of a classical and virtual knot}, 
Proceedings of the Conference on Geometric and Combinatorial Group Theory, Part II (Haifa, 2000), 
Geom.\ Dedicata \textbf{95} (2002), 215--225

\bibitem{turaev}
V.~Turaev, 
\emph{Knotoids}, 
preprint available at \url{http://arxiv.org/abs/1002.4133}

\end{thebibliography}
\end{document}